\documentclass[11pt]{article}
\usepackage{CJK,amsmath,amsthm,amsfonts,amssymb,fancyhdr}
\usepackage[shortlabels]{enumitem}
\setlist[enumerate,1]{label={\upshape(\roman*)}}
\usepackage{color,soul,supertabular,longtable,verbatim,extarrows}
\usepackage{titlesec}%\UTF{FFFD}\UTF{FFFD}\UTF{FFFD}\UTF{FFFD}\UTF{FFFD}\UTF{00BD}\UTF{FFFD}\UTF{FFFD}\UTF{FFFD}\UTF{FFFD}\UTF{FFFD}%
\usepackage[dvipdfmx]{graphicx}
\usepackage[noadjust]{cite}
\usepackage{tikz}
\usepackage{booktabs,multirow}
\usepackage{authblk}
\usepackage{lipsum}
\usepackage{mathrsfs}
\usepackage[top=2.4cm,bottom=2.2cm,left=2.6cm,right=2cm]{geometry} % \UTF{04B3}\UTF{FFFD}\UTF{07FE}\UTF{FFFD}%

\newtheorem{theorem}{Theorem}[section]%\UTF{FFFD}\UTF{FFFD}\UTF{FFFD}\UTF{FFFD}\UTF{FFFD}\UTF{FFFD}\UTF{FFFD}\UTF{FFFD}\UTF{FFFD}\UTF{FFFD}\UTF{FFFD}\UTF{FFFD}\UTF{FFFD}\UTF{FFFD}\UTF{FFFD}\UTF{FFFD}\UTF{FFFD}\UTF{FFFD}%
\newtheorem{proposition}[theorem]{Proposition}
\newtheorem{lemma}[theorem]{Lemma}
\newtheorem{corollary}[theorem]{Corollary}
\theoremstyle{definition}
\newtheorem{definition}[theorem]{Definition}
\newtheorem{example}[theorem]{Example}
\newtheorem{remark}[theorem]{Remark}

\newtheorem{question}[theorem]{Question}

\DeclareMathOperator{\tr}{tr}
\DeclareMathOperator{\rk}{rank}
\DeclareMathOperator{\order}{order}
\newcommand{\bI}{\mathbf{I}}
\newcommand{\bJ}{\mathbf{J}}
\newcommand{\bj}{\mathbf{j}}
\newcommand{\bu}{\mathbf{u}}
\newcommand{\bv}{\mathbf{v}}
\newcommand{\br}{\mathbf{r}}
\newcommand{\be}{\mathbf{e}}
\newcommand{\sA}{\mathsf{A}}
\newcommand{\sD}{\mathsf{D}}
\newcommand{\sE}{\mathsf{E}}
\newcommand{\sL}{\mathsf{L}}
\newcommand{\sM}{\mathsf{M}}
\newcommand{\Z}{\mathbb{Z}}

\begin{document}
\title{\textbf{Maximality of Seidel matrices and switching roots of graphs}}
\author[a]{Meng-Yue Cao}
\author[b,c,]{Jack H. Koolen\footnote{Corresponding author.}}
\author[d]{Akihiro Munemasa}
\author[d]{Kiyoto Yoshino}
\affil[a]{\footnotesize{School of Mathematical Sciences, Beijing Normal University, 19 Xinjiekouwai Street, Beijing, 100875, PR China.}}
\affil[b]{\footnotesize{School of Mathematical Sciences, University of Science and Technology of China, 96 Jinzhai Road, Hefei, 230026, Anhui, PR China.}}
\affil[c]{\footnotesize{Wen-Tsun Wu Key Laboratory of CAS, 96 Jinzhai Road, Hefei, 230026, Anhui, PR China}}
\affil[d]{\footnotesize{Graduate School of Information Sciences, Tohoku University, 6-3-09 Aramaki-Aza-Aoba, Aoba-ku, Sendai, 980-8579, Japan}}
\date{}
\maketitle
\newcommand\blfootnote[1]{%
\begingroup
\renewcommand\thefootnote{}\footnote{#1}%
\addtocounter{footnote}{-1}%
\endgroup}
\blfootnote{E-mail addresses: cmy1325@163.com (M.-Y. Cao), koolen@ustc.edu.cn (J.H. Koolen), munemasa@math.is.tohoku.ac.jp (A. Munemasa), kiyoto.yosino.r2@dc.tohoku.ac.jp (K. Yoshino).}

\renewcommand\abstractname{Abstract}
\begin{abstract}
In this paper, we discuss maximality of Seidel matrices with a fixed largest eigenvalue.
We present a classification of maximal Seidel matrices of largest eigenvalue 3, which gives a classification of maximal equiangular lines in a Euclidean space with angle $\arccos1/3$.
Motivated by the maximality of the exceptional root system $E_8$, we define strong maximality of a Seidel matrix, and show that every Seidel matrix achieving the absolute bound is strongly maximal.

 \end{abstract}

\section{Introduction}\label{sec:intro}
Throughout this paper, we consider only simple undirected graphs without loops.
For terminology which we do not define see~\cite{brouwer2011spectra,godsil2013}.
The Seidel matrix $S=S(G)$ of a graph $G$ is defined to be $S:=\mathbf{J-I}-2A$,
where $A:=A(G)$ is the adjacency matrix of $G$.
Alternatively, a Seidel matrix is a symmetric matrix with zero diagonal and all off-diagonal entries $\pm1$.

Seidel matrices are introduced in connection with equiangular lines in Euclidean spaces.
If $S$ is a Seidel matrix of a graph $G$, then we establish a close connection between these geometric objects by the use of the ``switching root'', which we introduce in the present paper.
Specifically, we consider the properties of Seidel matrices defined in the following.
Note that $\rk(W)$ denotes the rank of a matrix $W$.

\begin{definition}	\label{dfn:max}
Let $S$ be a Seidel matrix with largest eigenvalue $\lambda$.
We say that $S$ is \emph{maximal}
if there is no Seidel matrix $S'$ satisfying the following conditions~\ref{dfn:max:0}, \ref{dfn:max:1} and~\ref{dfn:max:2}.
We say that $S$ is \emph{strongly maximal}
if there is no Seidel matrix $S'$ satisfying~\ref{dfn:max:0} and~\ref{dfn:max:1}.
\begin{enumerate}[(1)]
	\item The largest eigenvalue of $S'$ equals $\lambda$.	\label{dfn:max:0}
	\item The Seidel matrix $S'$ contains $S$ as a proper principal submatrix.	\label{dfn:max:1}
	\item $\rk(\lambda\bI-S') = \rk(\lambda\bI-S)$.	\label{dfn:max:2}
\end{enumerate}
If $S$ is not strongly maximal, then we call $S$ \emph{extendable}.

We say that a graph $G$ is \emph{maximal}, \emph{strongly maximal} and \emph{extendable},
if $S(G)$ is maximal, strongly maximal and extendable, respectively.
\end{definition}

A set of lines in a Euclidean space is \emph{equiangular} if any pair of lines forms the same angle.
The \emph{rank} of a set of equiangular lines is the smallest dimension of Euclidean spaces into which these lines are isometrically embedded.
Denote by $N_\alpha(d)$ the maximum cardinality of a set of equiangular lines with angle $\arccos(\alpha)$ in dimension $d$,
and denote by $N^*_\alpha(r)$ that with angle $\arccos(\alpha)$ of rank $r$.
Then we have $N_\alpha(d) = \max_{r \leq d} N^*_\alpha(r)$.
Note that, if a Seidel matrix $S$ has largest eigenvalue $\lambda$, then there exist vectors whose Gram matrix equals $\lambda\bI-S$.
In this case, such vectors span equiangular lines with common angle $\arccos(1/\lambda)$, and the rank of $\lambda\bI-S$ equals that of these lines.
Note that $S$ is maximal if and only if the set of equiangular lines so obtained is saturated in the sense of \cite{Lin2020, Lin2020b}.
For example, $S := \bJ_4 - \bI_4$ is a Seidel matrix having largest eigenvalue $\lambda = 3$, and induces the set of  equiangular lines $\mathbb{R} \bu_1$, $\mathbb{R} \bu_2$, $\mathbb{R} \bu_3$ and $\mathbb{R} \bu_4$ with common angle $\arccos(1/3)$,
where
\begin{align*}
	&\bu_1 := (1,1,1,0,0,0)^\top/\sqrt{3},	\quad
	&&\bu_2 := (-1,0,0,1,1,0)^\top/\sqrt{3},	\quad \\
	&\bu_3 := (0,-1,0,-1,0,1)^\top/\sqrt{3},	\quad
	&&\bu_4 := (0,0,-1,0,-1,-1)^\top/\sqrt{3}.
\end{align*}
Namely, $|(\bu_i,\bu_j)| = 1/\lambda = 1/3$ holds for $i \neq j$.
Since $\bu_1$, $\bu_2$, $\bu_3$ and $\bu_4$ generate a $3$-dimensional $\mathbb{R}$-vector space by $\bu_1+\bu_2+\bu_3+\bu_4=0$, we have $N^*_{1/3}(3) \geq 4$.
In fact equality holds by Corollary~\ref{cor:1/3}, and hence this Seidel matrix $S$ is maximal.
Note that Lin and Yu~\cite{Lin2020b} provided several saturated sets of equiangular lines, or equivalently maximal Seidel matrices.

Lemmens and Seidel determined $N_{1/3}(d)$ for every positive integer $d$ in~\cite[Theorem~4.5]{lemmens1973}.
In particular, it asserts that $N_{1/3}(7) = \cdots = N_{1/3}(14) = 28$.
By~\cite[Theorem~4]{glazyrin2018}, every set of equiangular lines in $\mathbb{R}^n$ $(n \leq 11)$ of cardinality $28$ with common angle $\arccos(1/3)$ is contained in a $7$-dimensional subspace.
Namely, $N^*_{1/3}(n) < 28 = N^*_{1/3}(7)$ holds for every $n \in \{8,\ldots,11\}$.
Moreover, it has been proved by Lin and Yu that $N^*_{1/3}(8) = 14$~\cite[Proposition~5.2]{Lin2020} and the set of equiangular lines of rank $8$ and cardinality $14$ with angle $\arccos(1/3)$ is unique~\cite[Remark on p.\ 14]{Lin2020}.
In Section~\ref{sec:3}, we present Theorem~\ref{thm:3} as the first main result, which determines maximal and strongly maximal graphs with largest Seidel eigenvalue $3$.
This immediately implies a more precise and general result as Corollary~\ref{cor:1/3},
which determines the sets of equiangular lines with angle $\arccos(1/3)$ of a given rank $r$ and cardinality $N^*_{1/3}(r)$.

Let $S$ be a Seidel matrix of order $n$ with largest eigenvalue $\lambda$, and let $r = \rk(\lambda \bI - S)$.
It is known that the absolute bound $n \leq r(r+1)/2$ can be achieved if $r \in \{2,3,7,23\}$.
Moreover, a Seidel matrix which attains this bound is unique up to switching for each rank $r \in \{2,3,7,23\}$
(see~Theorem~\ref{thm:3} for $r=7$ and \cite[Theorem~A]{goethals1975} for $r=23$).
The second main result is Theorem~\ref{regular2}, which shows that a Seidel matrix attaining the absolute bound is strongly maximal.
In addition, it follows from Theorem~\ref{thm:3} that a strongly maximal graph with largest Seidel eigenvalue $\lambda = 3$, which attains the absolute bound for $r=7$, is unique up to switching.
An analogue is verified for each $(\lambda,r) \in \{(2,2),(\sqrt{5},3)\}$ in Proposition~\ref{prop:2,5}.
Hence we suspect that  the disjoint union of the McLaughlin graph and $K_1$, which attains the absolute bound for $r=23$, is a unique strongly maximal graph with largest Seidel eigenvalue $5$ up to switching.

This paper is organized as follows.
In Section~\ref{sec:sw}, we introduce the ``switching root'' and provide a theorem that shows a relationship between the eigenvalues of graphs and those of Seidel matrices.
In Section~\ref{sec:3}, we classify the maximal Seidel matrices with largest eigenvalue $3$.
In Section~\ref{sec:rank}, we prepare for the next section.
In Section~\ref{sec:strong_max}, we prove that a graph which attains the absolute bound is strongly maximal, and discuss their uniqueness.
In Section~\ref{sec:ex}, we discuss the existence of strongly maximal graphs whose largest Seidel eigenvalue is less than $3$, and also provide two families of infinitely many strongly maximal graphs with unbounded largest Seidel eigenvalue.
\section{Switching root}	\label{sec:sw}
Let $G=(V,E)$ be a graph.
For a subset $U$ of $V$, the graph $G^U=(V,E^U)$ is the graph obtained as follows:
\begin{equation}\nonumber
x\sim y \ \text{in} \ G^U \ \text{if}\ \left\{
 \begin{array}{ll}
x\sim y \text{ in $G$ and }x,y\in U ,\\
x\sim y \text{ in $G$ and }x,y\in V\setminus U ,\\
x\not\sim y \text{ in $G$ and }x\in U,\ y\in V\setminus U.
 \end{array}
 \right.
\end{equation}
We say that $G^U$ is the graph obtained from $G$ by switching with respect to $U$.
Note $G^U=G^{V\setminus U}$. Note further that the spectrum of $S(G^U)$ is equal to the spectrum of $S(G)$ for all $U\subseteq V$, as they are similar.
The graphs $G$ and $G^U$ are called switching equivalent.
Switching equivalence is an equivalence relation, since
$(G^U)^W = G^{U \Delta W}$ where $\Delta$ denotes symmetric difference.
The equivalence class $[G]$ of $G$, called the switching class of $G$,
is the set $\{G^U \mid U \subseteq V\}$.
\begin{definition}\label{dfn:swroot}
Let $G=(V,E)$ be a graph having largest Seidel eigenvalue $2\theta -1$, where $\theta$ is a positive real number.
Let $\{\mathbf{\alpha}^{(x)}\mid x\in V\}$ be the set of vectors in $\mathbb{R}^m$ for some positive integer $m$ such that the inner product $(\mathbf{\alpha}^{(x)},\mathbf{\alpha}^{(y)})$ satisfies
\begin{equation}\label{eq:Gram}
(\mathbf{\alpha}^{(x)},\mathbf{\alpha}^{(y)})=(A(G)+\theta \mathbf{I})_{xy}\quad(x,y\in V).
\end{equation}
A vector
%Furthermore,  let
$\mathbf{r}$ is called a {\em switching root} of $G$ if
%be a vector
\begin{enumerate}[(1)]
\item $(\mathbf{r},\mathbf{r})=2$ and
\item $(\mathbf{r}, \mathbf{\alpha}^{(x)})=1$ for all vertices $x$ of $G$.
\end{enumerate}
% and  are satisfied. We call $\mathbf{r}$ the
\end{definition}
One could consider a configuration of vectors $\alpha^{(x)}$ ($x\in V$) and a switching root $\mathbf{r}$ for an arbitrary positive real number $\theta$, in Definition~\ref{dfn:swroot}.
The existence of such a configuration is equivalent to the condition that the matrix $B_\theta(G)$ defined in Definition~\ref{dfn:B} below is positive semidefinite. The following theorem justifies that the choice of $\theta$ in Definition~\ref{dfn:swroot} is the optimal one.

The reason for the name ``switching root'' is the following.
Let $U\subseteq V(G)$ and let $G^U$ be the graph obtained from $G$ by switching with respect to $U$. Consider the vectors $\mathbf{\beta}^{(x)}$ defined as follows:
$ \mathbf{\beta}^{(x)}:= \mathbf{\alpha}^{(x)}$ if $x \in V(G) \setminus U$ and $ \mathbf{\beta}^{(x)}:= \mathbf{r} - \mathbf{\alpha}^{(x)}$ if $x \in U$. Then, we have $(\mathbf{\beta}^{(x)},\mathbf{\beta}^{(y)})=(A(G^U)+\theta \mathbf{I})_{xy}$ for $x,y\in V$.

\begin{definition}\label{dfn:B}
Let $\theta$ be a positive real number and let $G$ be a graph.
For any real number $t$, we define the matrix $B^{(t)}_{\theta}(G)$ as
\[
B^{(t)}_{\theta}(G):=
        \left(
        \begin{array}{cc}
        A(G)+\theta\mathbf{I} & \mathbf{j}\\
        \mathbf{j}^T & t \\
        \end{array}
        \right),
\]
where $\mathbf{j}$ denotes the all-ones vector.
In particular, we set $B_{\theta}(G) := B^{(2)}_{\theta}(G)$.
Moreover we assume that $-\theta$ is at least the smallest eigenvalue of $A(G)$, and then define
\[p(G) := \min\{ t \in \mathbb{R} \mid
B^{(t)}_{\theta}(G)\text{ is positive semi-definite}\}.\]
\end{definition}

Note that for a graph $G$ having at least one vertex, the value $p(G)$ is positive.
For every $t \neq 0$, we have
\begin{align}	\label{eq:block diag B^t}
	\begin{pmatrix}
	\mathbf{I} & -\frac{1}{t}\mathbf{j}\\
	0 & 1
	\end{pmatrix}
	B^{(t)}_\theta(G)
	\begin{pmatrix}
	\mathbf{I} & 0\\
	-\frac{1}{t}\mathbf{j}^\top & 1
	\end{pmatrix}
	=
	\begin{pmatrix}
	A(G)+\theta\mathbf{I}-\frac{1}{t} \mathbf{J} & 0\\
	0 & t
	\end{pmatrix}.
\end{align}
For $t=2$, this together with $2(A(G)+\theta\mathbf{I})-\mathbf{J} = (2\theta-1)\mathbf{I}-S(G)$ implies the following theorem.
\begin{theorem}\label{main}
	Let $\theta$ be a positive real number and $G$ be a graph.
	%Let $B := B_\theta(G)$ as defined above.
	Then the following two statements are equivalent:
	\begin{enumerate}[(1)]
		\item $S(G)$ has largest eigenvalue at most $2 \theta-1$;	\label{main:1}
		\item $B_\theta(G)$ is positive semi-definite.	\label{main:2}
	\end{enumerate}
	If one of the equivalent conditions~\ref{main:1} and \ref{main:2} holds,
	then $\rk((2\theta-1)\bI - S(G)) + 1 = \rk(B_\theta(G))$, $A(G)$ has least eigenvalue at least $-\theta$, and $p(G) \leq 2$  holds.
\end{theorem}

The \emph{cone} over a graph $G$, denoted by $\tilde{G}$, is defined to be the graph obtained by adding a new vertex to $G$ and connecting it to all the vertices of $G$.
\begin{corollary}	\label{cor:main}
	For every graph $G$ of order $n$, the following are equivalent:
	\begin{enumerate}[(1)]
		\item The graph $G$ has largest Seidel eigenvalue (resp.\ at most) $3$.	\label{cor:main:1}	
		\item The cone $\tilde{G}$ over $G$ has smallest eigenvalue (resp.\ at least) $-2$. 	\label{cor:main:2}
	\end{enumerate}
	If $S(G)$ has largest Seidel eigenvalue at most $3$, then $\rk(3\mathbf{I}-S(G))+1 = \rk(A(\tilde{G})+2\bI)$.
\end{corollary}

\section{Classification of maximal Seidel matrices with largest eigenvalue~$3$}	\label{sec:3}
We prove the following theorem at the end of this section, which gives some maximal graphs (up to switching) with largest Seidel eigenvalue $3$ and also a strongly maximal one.
Note that we denote by $G+H$ the disjoint union of two graphs $G$ and $H$.

\begin{theorem}	\label{thm:3}
	Let $G$ be a graph of order $n$ having largest Seidel eigenvalue $3$ with multiplicity $m$.
	Assume that $G$ is maximal.
	Then it is switching equivalent to one of the following.
	\begin{enumerate}[(1)]
		\item $L(K_{5})$ and $L(K_{2,4})$ if $n-m =5$.
		\item $L(K_{6})+K_1$ and $L(K_{2,5})$ if $n-m =6$.		
		\item $L(K_8)$ if $n-m=7$.				
		\item $L(K_{2,n-m-1})$ if $n-m = 3,4$ or $n-m \geq 8$.
	\end{enumerate}		
Furthermore, if $G$ is strongly maximal, then it is switching equivalent to $L(K_8)$.
\end{theorem}

Since $N^*_{1/3}(r)$ is the maximum order of a Seidel matrix $S$ with largest eigenvalue $3$ and $\rk (3\bI-S)=r$, Theorem~\ref{thm:3} implies
the following corollary.

\begin{corollary}	\label{cor:1/3}
	Let $r$ be an integer at least $3$.
	Then $N_{1/3}^*(r)$ equals $10$ if $r=5$, $16$ if $r=6$, $28$ if $r=7$ and $2(r-1)$ otherwise.
	More precisely, an arbitrary set of equiangular lines with common angle $\arccos(1/3)$ of rank $r$ and cardinality $N_{1/3}^*(r)$ is induced by the Seidel matrix of a graph switching equivalent to $L(K_5)$ if $r=5$, $L(K_6)+K_1$ if $r=6$, $L(K_8)$ if $r=7$ and $L(K_{2,r-1})$ otherwise.
\end{corollary}

\begin{definition}	\label{dfn:Lambda}
	For a graph $G$ whose cone $\tilde{G}$ has smallest eigenvalue at least $-2$,
	we define $\Lambda(G)$ to be the lattice generated by vectors of which Gram matrix equals $A(\tilde{G}) + 2\bI$. And we denote by $\rk \Lambda(G)$ the rank of $\Lambda(G)$, which equals $\rk (A(\tilde{G}) + 2\bI)$.
\end{definition}

Corollary~\ref{cor:main} implies the following.

\begin{lemma}	\label{lem:rk}
	For a graph $G$ with largest Seidel eigenvalue at most $3$, $\rk(3\bI-S(G)) +1 = \rk \Lambda(G)$.
\end{lemma}

A vector of norm $2$ is called a \emph{root}, and an integral lattice generated by roots is called a \emph{root lattice}.
If $G$ is a graph whose cone has smallest eigenvalue at least $-2$, then $\Lambda(G)$ is an irreducible root lattice.
It is known that the irreducible root lattices are enumerated up to isometry as follows:
\begin{align*}
	\sA_n &:= \{ \bv \in \Z_{n+1} \mid (\bv,\bj) = 0\} \quad (n \in \Z_{\geq 1}),\\
	\sD_n &:= \{ \bv \in \Z_n \mid (\bv,\bj) \in 2\Z \} \quad (n \in \Z_{\geq 4}),\\
	\sE_8 &:= \sD_8 \sqcup \left( \bj/2 + \sD_8 \right),\\
	\sE_7 &:= \{ \bv \in \sE_8 \mid (\bv,\be_1-\be_2)=0\}, \\
	\sE_6 &:= \{ \bv \in \sE_8 \mid (\bv,\be_1-\be_2) = (\bv,\be_2-\be_3) = 0\} .
\end{align*}
Here $\be_i$ denotes the vector of which the $i$-th entry is $1$ and the others are $0$.
We say that $\sD_n$ $(n \in \Z_{\geq 4})$ is a \emph{root lattice of type $D$},
and $\sE_n$ $(n=6,7,8)$ is a \emph{root lattice of type $E$}.
A large number of non-isomorphic connected graphs can give rise to the same irreducible root lattice. However, there is a natural way to recover a switching class of a graph from each irreducible root lattice.
\begin{definition}	\label{dfn:[L]}
Let $\sL$ be an irreducible root lattice.
The switching class, denoted by $[\sL]$, is defined to be the \emph{switching class} $[L]$ of a graph $L$ chosen as follows:  	
Let $\br$ be a root in $\sL$,
and $N$ the set of roots $\bv$ in $\sL$ with $(\br,\bv) = 1$.
Choose a subset $X \subset N$ of cardinality $|N|/2$ which has no roots $\bu$ and $\bv$ with $\bu = \br-\bv$.
Let $L$ be a graph such that $A(L) + 2\bI$ coincide with the Gram matrix of $X$.
\end{definition}
In this definition, we note that for two distinct roots $\bu$ and $\bv$ in $X$, the inner product $(\bu,\bv)$ is either $0$ or $1$, and the desired graph $L$ exists.
Since $\br$ is the switching root of $L$, the argument after Definition~\ref{dfn:swroot} implies that $[L]$ does not depend on the choice of $X$.
In addition, since the automorphism group of $\sL$ acts transitively on the roots in $\sL$, we see that $[L]$ does not depend on the choice of $\br$, and that $[\sL]$ is well-defined.
Note that the vectors with Gram matrix $A(L)+2\bI$ may not generate $\sL$, although that of the cone $\tilde{L}$ over $L$ always do.
Next we describe the switching class $[\sL]$ for each root lattice $\sL$.

\begin{lemma}	\label{lem:SC}
	The following hold.
	\begin{enumerate}[(1)]
		\item $[\sA_n] = [K_{n-1}]$ for each $n \in \Z_{\geq 1}$.	\label{lem:SC:1}
		\item $[\sD_n] = [L(K_{2,n-2})]$ for each $n \in \Z_{\geq 4}$.	\label{lem:SC:2}
		\item $[\sE_ 8] = [L(K_8)]$,	
		 $[\sE_ 7] = [L(K_6) + K_1]$, and
		 $[\sE_ 6] = [L(K_5)]$.		\label{lem:SC:3}
	\end{enumerate}
	In particular, all the graphs in switching classes $[\sD_n]$ $(n \geq 4)$ and $[\sE_n]$ $(n = 6,7,8)$ have largest Seidel eigenvalue $3$,
	and those in $[\sA_n]$ $(n \geq 1)$ have largest Seidel eigenvalue $1$.
\end{lemma}
\begin{proof}
	Throughout this proof, we firstly fix a (switching root) $\br$, secondly choose a subset $X$ as in Definition~\ref{dfn:[L]}, and determine the switching class $[L]$.

	First we show \ref{lem:SC:1}.
	Let $\br := \be_1-\be_2$, and $X := \{ \be_1 - \be_i \mid i = 3,\ldots,n+1\}$.
	Then since the Gram matrix of $X$ coincide with $A(K_{n-1}) + 2 \bI_{n-1}$, we have $[\sA_n] = [K_{n-1}]$.
	
	Next we show \ref{lem:SC:2}.
	Let $\br := \be_1 + \be_2$, and $X := \{ \be_i+ \be_j \mid i=1,2 \text{ and } j = 3,\ldots,n \}$.
	Then the matrix whose columns are vectors in $X$ is equal to the incidence matrix of $K_{2,n-2}$.
	Hence the Gram matrix of $X$ coincide with $A(L(K_{2,n-2})) + 2\bI_{n}$, we have $[\sD_n] = [L(K_{2,n-2})]$.
	
	We show that $[\sE_ 8] = [L(K_8)]$.
	Let $\br := \bj/2$, and
	\[
		X
		:= \{ \bv \in \sD_8 \mid (\br,\bv) = 1 \}
		= \{ \be_i + \be_j \mid 1 \leq i < j \leq 8 \}.
	\]		
	By an argument similar to that to show \ref{lem:SC:2}, we obtain the desired result.
	
	We show that $[\sE_ 7] = [L(K_7)]$.
	Let $\br := \bj/2$, and
	\begin{align*}
		X
		:= \{ \bv \in \sD_8 \mid (\br,\bv) = 1 \text{ and } (\bv,\be_1-\be_2) = 0\}	
		= \{ \be_1 + \be_2 \} \cup \{ \be_i + \be_j \mid 3 \leq i < j \leq 8 \}.
	\end{align*}
	This implies the desired result as well.
	
	We can verify that $[\sE_ 6] = [L(K_5)]$ by letting $\br := \bj/2$ and
	\begin{align*}
		X
		:= &\{ \bv \in \sD_8 \mid (\br,\bv) = 1 \text{ and } (\bv,\be_1-\be_2)  = (\bv,\be_2-\be_3) = 0\}	\\
		= &\{ \be_i + \be_j \mid 4 \leq i < j \leq 8 \}.
	\end{align*}
	
	Finally, let $\sL$ be an irreducible root lattice of type $D$ or $E$, and fix a graph $L \in [\sL]$.
	Then by Definition~\ref{dfn:[L]}, there exists a switching root of $L$ in $\sL$.
	Hence the cone $\tilde{L}$ has smallest eigenvalue at least $-2$.
	Since $[\sL]$ has been revealed above,
	we obtain \[\rk( A(\tilde{L}) + 2\bI ) \leq \rk \sL < \order \tilde{L}.\]
	This means that $\tilde{L}$ has smallest eigenvalue $-2$.
	By Corollary~\ref{cor:main}, the largest Seidel eigenvalue of $L$ equals $3$.	
	Since the Seidel spectrum of $K_{n-1}$ is $\{1,[-n+2]^{n-2}\}$, we obtain the desired conclusion.
\end{proof}

\begin{lemma}	\label{lem:reduc_to_L}
	For a graph $G$ with largest Seidel eigenvalue at most $3$,
	there exists a supergraph $L \in [\Lambda(G)]$ of $G$.
	In particular, $\Lambda(L)=\Lambda(G)$.
\end{lemma}
\begin{proof}
	Set $\sL := \Lambda(G)$.
	Let $\br$ be the root in $\sL$ corresponding to the vertex of $\tilde{G}$ added to $G$.
	Let $N$ be the set of roots in $\sL$ with $(\br,\bu) = 1$.
	Then we can choose a subset $X \subset N$ with $2|X| = |N|$ such that $X$ contains the roots in $\Lambda(G)$ corresponding to the vertices of $G$.
	By Definition~\ref{dfn:[L]}, we see that $G$ is an induced subgraph of some graph $L$ in $[\sL]$.
	Next we obtain
	$
		\sL		
		= \Lambda(G)
		\subset \Lambda(L)
		\subset \sL
	$
	as desired.
\end{proof}

\begin{lemma}	\label{lem:reduction}
	Let $\sL$ be an irreducible root lattice of type $D$ or $E$, and let $L \in [\sL]$ with $\sL = \Lambda(L)$.
	Then $L$ is maximal (resp.\ strongly maximal)
	if and only if
	there is no irreducible root lattice $\sM$ of type $D$ or $E$ satisfying the following~\ref{lem:reduction:1} and~\ref{lem:reduction:2} (resp.\ only the following~\ref{lem:reduction:1}).
	\begin{enumerate}[(1)]
		\item	\label{lem:reduction:1}
			The lattice $\sM$ properly containing $\sL$ up to isometry.			
		\item	\label{lem:reduction:2}
			The rank of $\sM$ equals that of $\sL$.
	\end{enumerate}
\end{lemma}
\begin{proof}
	By Lemma~\ref{lem:SC}, the largest Seidel eigenvalue of $L$ is $3$.
	Assume there exists a supergraph $H$ of $L$ with largest Seidel eigenvalue $3$.
	Applying Lemma~\ref{lem:reduc_to_L} with $G := H$ and setting $\sM := \Lambda(H)$, we have a supergraph $M \in [\sM]$ of $H$
	with $\sM = \Lambda(M)$.
	Then the largest Seidel eigenvalue of $M$ is at least that of $L$, and at most $3$ by Lemma~\ref{lem:SC}.
	Hence we see that $M$ has largest Seidel eigenvalue $3$, and $\sM$ is of type $D$ or $E$.	
	Without loss of generality, we may assume that $\sM$ contains $\sL$.
	Noting that $\sL = \Lambda(L)$ and $\sM = \Lambda(M)$,
	we see that
	$L = M$ if and only if $\sL = \sM$.
	This gives the desired equivalent condition for the graph $L$ to be strongly maximal.

	Finally, Lemma~\ref{lem:rk} implies that
	\[
		\rk (3\bI -S(L)) + 1= \rk \sL \quad \text{and} \quad \rk (3\bI -S(M))+1 = \rk \sM.
	\]
	Hence Condition~\ref{lem:reduction:2} is equivalent to $\rk (3\bI -S(L)) = \rk (3\bI -S(M))$.
	This gives the desired result.
\end{proof}

\begin{proof}[Proof of Theorem~\ref{thm:3}]
	Let $G$ be a maximal graph with largest Seidel eigenvalue $3$.
	By applying Lemma~\ref{lem:reduc_to_L} and setting $\sL := \Lambda(G)$,
	we obtain a supergraph $L \in [\sL]$ of $G$ with $\sL=\Lambda(L)$.
	Moreover, Lemma~\ref{lem:SC} implies that the largest Seidel eigenvalue of $L$ is $3$ and that $\Lambda(G)$ is of type $D$ or $E$.
	Since
	\[
		\rk( 3\bI - S(L) ) + 1 = \rk \Lambda(L) = \rk \sL = \rk \Lambda(G) = \rk( 3\bI - S(G) ) + 1,
	\]
	by Lemma~\ref{lem:rk},
	the maximal graph $G$ coincide with $L$.
	The following relations among root lattices are well known.
	\begin{align*}
		&\sD_4 \subset \sD_5 \subset \cdots, \sE_6 \subset \sE_7 \subset \sE_8,\\
		&\sD_6 \not\subset \sE_6,
		\sD_7 \not\subset \sE_7,
		\sD_8 \subset \sE_8,\\
		&\sE_n \not\subset \sD_{n'} \text{ for } n \text{ and } n'.
	\end{align*}	
	Therefore, Lemma~\ref{lem:reduction} implies the desired conclusion.
\end{proof}

\section{On the rank of $B_{\theta}(G)$}\label{sec:rank}
%\noindent
%From now on,  we assume that $\theta$ is a positive integer and assume that the largest eigenvalue of $S(G)$ is equal to $2\theta-1$.
In the next lemma, we show that as long as there exists an eigenvector of $S(G)$ for the eigenvalue $2\theta-1$ which is not orthogonal to the all-ones vector, $\rk(A(G)+\theta \mathbf{I})=\rk(B_\theta(G))$ holds.

\begin{lemma}\label{mainev}
Let $G=(V,E)$ be a graph having largest Seidel eigenvalue $2 \theta -1$.
Let $\{\mathbf{\alpha}^{(x)}\mid x\in V\}$ be the set of vectors
satisfying \eqref{eq:Gram}.
% such that $(\mathbf{\alpha}^{(x)},\mathbf{\alpha}^{(y)})
% =(A+\theta \mathbf{I})_{xy}$ for $x,y\in V$.
%For every eigenvector
Suppose that $\mathbf{v}$ is an eigenvector
of $S$ belonging to the eigenvalue $2\theta - 1$
and $(\mathbf{v},\mathbf{j})\neq0$.
Then
% which is not orthogonal to $\mathbf{j}$,
the vector
\[\mathbf{r} := \frac{2 }{(\mathbf{v},\mathbf{j})}
\sum_{x\in V} \mathbf{v}_x\mathbf{\alpha}^{(x)}
\]
is a switching root of $G$.
In particular, $\rk (B_{\theta}(G))= \rk (A(G) + \theta \mathbf{I})$.

\end{lemma}
\begin{proof}
Let $S = S(G)$ and $A=A(G)$.
%Since $(2\theta  -1)\mathbf{I} - S$ is an integral matrix, we may suppose that $2/(\mathbf{v}, \mathbf{j})$ is a positive integer.
Let $N$ be the matrix whose columns are all the vectors $\alpha^{(x)}$.
Then $A + \theta \mathbf{I} = N^\top N$, and $\mathbf{r} = 2N \mathbf{v}/(\mathbf{v},\mathbf{j})$.
Since $2 (A + \theta \mathbf{I}) = ( (2 \theta - 1)\mathbf{I} - S ) + \mathbf{J}$, we have, for any vector $\mathbf{u}$,
\begin{align*}
	2(N \mathbf{u})^\top (N \mathbf{v}) = 2\mathbf{u}^\top ( A + \theta \mathbf{I} )\mathbf{v} = \mathbf{u}^\top( (2 \theta - 1 )\mathbf{I}) - S ) + \mathbf{J})\mathbf{v} = \mathbf{u}^\top \mathbf{J} \mathbf{v} = (\mathbf{u},\mathbf{j})(\mathbf{v},\mathbf{j}).
\end{align*}
Letting
$\mathbf{u} := \mathbf{v}$, we obtain
$(\mathbf{r},\mathbf{r})=2$. Similarly,
letting
$\mathbf{u} := \mathbf{e}_x$ for a vertex $x$, where $\mathbf{e}_x$ denotes the characteristic vector of $\{x\}$ indexed by $V$,
we have $N \mathbf{u} = \alpha^{(x)}$ and
$(\mathbf{u},\mathbf{j}) = 1$.
Hence $(\mathbf{r},\alpha^{(x)})=1$ holds.
\end{proof}

Seidel matrices with exactly two eigenvalues are known as regular two-graphs~\cite[Section 11.6]{godsil2013}.
Now we will look at graphs such that $S(G)$ has exactly two distinct eigenvalues $2\theta-1$ and $2\tau-1$ with respective multiplicities $m_{\theta}$ and $m_{\tau}$.
Then we have
\begin{align}	\label{eq:Seidel rel}
	-(2\theta-1)(2 \tau -1) = n-1 \text{ and } (2\theta-1)m_\theta + (2\tau-1)m_\tau =  0,
\end{align}
where $n$ is the order of $G$.
Simple examples are complete graphs and their complements.
It is easy to see that an arbitrary graph $G$ having largest Seidel eigenvalue at most $1$ is switching equivalent to the complete graph.
Thus, in order to avoid the trivial case, we assume that $\theta > 1$ if necessary.

A graph $G$ of order $n$ is said to be strongly regular with parameters $(n,k,a,c)$, if it is $k$-regular, every pair of adjacent vertices has $a$ common neighbours, and every pair of distinct nonadjacent vertices has $c$ common neighbours.

\begin{lemma}\label{srg}
Let $G$ be a graph of order $n \geq 2$ having two distinct Seidel eigenvalues $2\theta -1$ and $2\tau-1$, with respective multiplicities $m_{\theta}$ and $m_{\tau}$, where $\theta > 1 > \tau$.
Let $H \in [G]$ be such that $\rk (B_{\theta}(H)) \neq \rk (A(H) + \theta \mathbf{I})$. Then $H$ is a strongly regular graph with spectrum
$\{\frac{n - 2\tau}{2} , [-\tau]^{m_{\tau}-1} , [-\theta]^{m_{\theta}}\}$.
\end{lemma}

\begin{comment}
\begin{proof}
By Lemma \ref{mainev}, we know that the eigenvalue $2\theta-1$ of $S(H)$ is not a main eigenvalue. Hence $A(H)$ has only one main eigenvalue. This means that $H$ is regular and that
$A(H)$ has eigenvalue $-\theta$ with multiplicity at least $m_{\theta}$, and eigenvalue $-\tau$ with multiplicity $m_{\tau}-1$. As the sum of the eigenvalues of $A(H)$ is equal to zero, $m_{\theta} + m_{\tau} = n$ and $m_{\theta} (2\theta-1) + m_{\tau}(2\tau -1) =0$, we obtain the spectrum of $A(H)$. As $\theta \neq 1$, we see that the three eigenvalues
$\frac{[n - 2\tau}{2}]^1 , -\tau , -\theta$ are distinct. Now $H$ is strongly regular as it has three distinct eigenvalue, regular and the largest eigenvalue is simple, by \cite[Lemma 10.2.1]{godsil2013}.
\end{proof}
\end{comment}

\begin{proof}
By Lemma \ref{mainev}, $\rk (B_{\theta}(H))= \rk (A(H) + \theta \mathbf{I})$ implies that the all-ones vector $\mathbf{j}$ is orthogonal to the eigenspace of $S(H)$ for the eigenvalue $2\theta -1$.
Then $\mathbf{j}$ is an eigenvector of $S(H)$ belonging to $2\tau -1$,
and one of $A(H)$ belonging to the eigenvalue $(n-2\tau)/2$.
This means that $H$ is regular of valency $(n-2\tau)/2$.
Moreover we obtain the desired spectrum of $A(H)$.
Note that
$
	(n-2\tau)/2 > -\tau > -\theta.
$
Since the largest eigenvalue of $A(G)$ is simple, $G$ is connected.
If $m_\tau = 1$, then $\theta \in \{0, 1\}$ by~\eqref{eq:Seidel rel}.
Thus we may assume that $m_\tau \geq 2$, and then $A(H)$ has exactly three distinct eigenvalues.
Therefore the graph $H$ is strongly regular (see~\cite[Lemma 10.2.1]{godsil2013}).
\end{proof}

Gerzon showed the following bound for a Seidel matrix. This bound is usually called the absolute bound.
\begin{lemma}[{\cite[Theorem~11.2.1]{godsil2013}}]	\label{thm:abs_bound}
Let $G$ be a graph of order $n$ having largest Seidel eigenvalue $2\theta-1$.
Let $r = \rk(S(G) -(2\theta-1) \mathbf{I})$. Then $n \leq \frac{r(r+1)}{2}$.
\end{lemma}

If equality holds, then it is known (see \cite[p.\ 253]{godsil2013}) that $r \in \{2,3\}$ or $r = (2t+1)^2 -2$ holds,
where $t$ is a positive integer if $r > 3$.
Now we apply Lemma \ref{srg} to graphs with equality in the absolute bound.

\begin{theorem}\label{regular}
Let $G$ be a graph of order $n$ having largest Seidel eigenvalue $2\theta-1$.
Let $r = \rk(S(G) -(2\theta-1) \mathbf{I})$. Assume that $n = \frac{r(r+1)}{2}$. Then for all $H \in [G]$, we have $\rk (B_{\theta}(H))= \rk (A(H) + \theta \mathbf{I})$.
\end{theorem}
\begin{proof}
By \cite[Lemma 11.3.1]{godsil2013}, we know that $S(G)$ has exactly two distinct eigenvalues. Assume that $S(G)$ has eigenvalues $2\theta-1$ and $2\tau -1$ with respective multiplicities $m_{\theta}=n-r$ and $m_{\tau}=r$.
If there exists $H \in [G]$ such that the conclusion does not hold,
then by Lemma~\ref{srg},  $H$ is strongly regular, and
the eigenvalue $-\tau$ of $A(H)$ has multiplicity $m_\tau-1=r-1$.
By the absolute bound for strongly regular graphs (see \cite[p.\ 120]{brouwer2011spectra}), $n \leq (r-1)(r+2)/2$ holds.
This is a contradiction.
\end{proof}

\begin{remark}
The above result was shown for $r=7$ and $r=23$ by Koolen and Munemasa \cite{koolen}.
\end{remark}

\section{Strong maximality of graphs which attain the absolute bound}\label{sec:strong_max}
In this section, we prove that a graph which attains the absolute bound is strongly maximal.
Moreover, we discuss the uniqueness of strongly maximal graphs.

\begin{lemma}\label{p(RG)}
	Let $G$ be a $k$-regular graph of order $n$ whose smallest eigenvalue is $-\theta$.
	Then $p(G) = n/(k+\theta)$.
\end{lemma}
\begin{proof}
Assume that $t$ is a positive number.
By Equation~\eqref{eq:block diag B^t},
the matrix $B_\theta^{(t)}(G)$ is positive semi-definite if and only if so is $A(G) + \theta\mathbf{I}-\frac{1}{t}\mathbf{J}$.
Since $G$ is regular, the smallest eigenvalue of $A(G) + \theta\mathbf{I}-\frac{1}{t}\mathbf{J}$ is $0$ or $k+\theta-n/t$.
Hence the desired result follows.
\end{proof}

The next result gives a necessary and sufficient condition for $G$ to be extendable.

\begin{lemma}\label{extendable}
Let $G$ be a graph with largest Seidel eigenvalue $2\theta-1$.
Then $G$ is extendable if and only if there exists a graph $H$ in $[G]$ with $p(H) \leq 2 -\frac{1}{\theta}$.
\end{lemma}
\begin{proof}
	The graph $G$ is extendable if and only if there exists $H\in [G]$ such that the largest eigenvalue of $S( H + K_1)$ is at most $2\theta-1$, where $H + K_1$ is the disjoint union of $H$ and $K_1$.
	By Theorem~\ref{main}, $S(H + K_1)$ has largest eigenvalue at most $2\theta-1$ if and only if $B_{\theta}(H + K_1)$ is positive semi-definite.
	Since $B_{\theta}(H + K_1)$ is congruent to $B_{\theta}^{(2-\frac{1}{\theta})}(H) \oplus(\theta	)$,
	the desired conclusion follows.
\end{proof}	

Combining Lemma~\ref{p(RG)} and {\ref{extendable}, we obtain the following lemma.	
	
\begin{lemma}	\label{k-reg extendable}
Let $G$ be a $k$-regular graph with largest Seidel eigenvalue $2\theta-1$.
If $\frac{n}{k+\theta} \leq 2-\frac{1}{\theta}$, then $G$ is extendable.
\end{lemma}	

The converse of this lemma is false.
In fact, the triangular graph $T(7)$ is a counter example.
Indeed, $T(7)$ is a strongly regular graph with parameters $(21, 10, 5, 4)$ and distinct eigenvalues $10, 3, -2$, and we have $p(T(7)) = 1.75 > 2 - \frac{1}{2}$.
However, as $T(8)$ has largest Seidel eigenvalue $3$, we see that $T(7)$ is extendable.

\begin{comment}
\begin{lemma}\label{srg2}
Let $G$ be a graph of order $n \geq 2$ such that $S(G)$ has two distinct eigenvalues $2\theta -1$ and $2\tau-1$, with respective multiplicities $m_{\theta}$ and $m_{\tau}$, where $\theta > 1 >\tau$. Then $G$ is extendable if and only if there exists $H \in [G]$, such that $H$ is a strongly regular graph with spectrum
$\{\frac{n - 2\tau}{2} , [-\tau]^{m_{\tau}-1} , [-\theta]^{m_{\theta}}\}$.
\end{lemma}
\begin{proof}
If $G$ is extendable, then, by Lemma \ref{extendable}, there exists a graph $H \in [G]$ such that $p(H) \leq 2-\frac{1}{\theta}$. This, in particular, implies $\rk(B_{\theta}(G)) \neq
\rk(A(H) + \theta \mathbf{I})$. By Lemma \ref{srg}, we have that $H$ must be a strongly regular graph with spectrum
$\{\frac{n - 2\tau}{2} , [-\tau]^{m_{\tau}-1} , [-\theta]^{m_{\theta}}\}$. As $H$ is a $k$-regular graph on $n$ vertices with $k =\frac{n - 2\tau}{2}$, then the value of $p(H)$ is determined by setting
 the determinant of the matrix
 \begin{equation*}Q=
        \left(
        \begin{array}{cc}
        k+ \theta & 1\\
        n & p \\
        \end{array}
        \right).
        \end{equation*}
 to   zero, as we are only interested in some eigenvector corresponding to eigenvalue $0$ for   $B_{\theta}(H) $ with  non-zero last coordinate.
By forward calculations (for details, see \cite[p.\ 257--258]{godsil2013}), it follows that $p(H) = 2 -\frac{1}{\theta}$.
This shows the lemma.
\end{proof}
\end{comment}

\begin{lemma}	\label{srg3}
	Let $G$ be a graph of order $n$ having two distinct Seidel eigenvalues $2\theta -1$ and $2\tau-1$, where $\theta > 1 >\tau$.
	Then the following are equivalent:
	\begin{enumerate}[(1)]
		\item $G$ is extendable,	\label{srg3:1}
%		\item There exists $H \in[G]$ such that $p(H) \leq 2-\frac{1}{\theta}$;	\label{srg3:2}
		\item There exists $H \in[G]$ such that $p(H) < 2$,	\label{srg3:3}			
		\item There exists $H \in[G]$ such that $\rk (B_\theta(H)) \neq \rk(A(H)+\theta \mathbf{I})$,	\label{srg3:5}
		\item There exists $H \in[G]$ such that $H$ is a strongly regular graph with degree $k=(n-2\tau)/2$.	\label{srg3:4}
	\end{enumerate}
\end{lemma}

\begin{proof}
	By Lemma~\ref{extendable}, \ref{srg3:1} implies \ref{srg3:3}.
	Suppose that \ref{srg3:3} is satisfied.
	Fix a graph $H$ such that $p(H) < 2$.
	Let $p := p(H)$.
	If the image of $A(H)+\theta \mathbf{I}$ does not contain the all-ones vector $\mathbf{j}$, then \ref{srg3:5} holds.
	Otherwise we may suppose that there exists a vector $\mathbf{b}$ such that $(A(H) + \theta\mathbf{I} )\mathbf{b} = \mathbf{j}$.
	Then $B_\theta(H)$ is congruent to
%	\begin{align*}
%		B_\theta(H)
%		=
%		\begin{pmatrix}
%			A(H) + \theta \mathbf{I} & \mathbf{j} \\
%			\mathbf{j} & p(H)
%		\end{pmatrix}
%		+
%		\begin{pmatrix}
%			\mathbf{0} & \mathbf{0} \\
%			\mathbf{0} & 2-p(H)
%		\end{pmatrix}
%		\begin{pmatrix}
%			A(H) + \theta \mathbf{I} & \mathbf{0} \\
%			\mathbf{0} & 2
%		\end{pmatrix}
%	\end{align*}
	\begin{align*}
		\begin{pmatrix}
			A(H) + \theta \mathbf{I} & \mathbf{0} \\
			\mathbf{0} & 2-(\mathbf{b},\mathbf{j})
		\end{pmatrix}
		=
		\begin{pmatrix}
			A(H) + \theta \mathbf{I} & \mathbf{0} \\
			\mathbf{0} & -(\mathbf{b},\mathbf{j}) + p
		\end{pmatrix}
		+
		\begin{pmatrix}
			\mathbf{0} & \mathbf{0} \\
			\mathbf{0} & 2-p
		\end{pmatrix}.
	\end{align*}
	Since the first term is congruent to $B_\theta^{(p)}(H)$, we obtain that $ -(\mathbf{b},\mathbf{j}) + p \geq 0$.
	Hence, by the assumption $p< 2$, we have $2-(\mathbf{b},\mathbf{j}) > 0$.
	This means that  \ref{srg3:5} holds.
	By Lemma~\ref{srg} \ref{srg3:5} implies \ref{srg3:4}.
	Finally, we suppose that \ref{srg3:4} is satisfied.
	By~\eqref{eq:Seidel rel}, we have
	\begin{align*}
		k = \frac{n-2\tau}{2}
		= \frac{(n-1) - (2\tau-1)}{2}
		= \frac{\theta(n-1)}{ 2\theta -1 }.
	\end{align*}
	This together with Lemma~\ref{p(RG)} imply that
	\begin{align*}
		\frac{n}{k+\theta}
		= \frac{n(2\theta-1)}{\theta(n-1) + \theta (2\theta -1 )}
		\leq \frac{n(2\theta-1)}{\theta(n-1) + \theta}
		= 2 - \frac{1}{\theta}.
	\end{align*}
	By Lemma~\ref{k-reg extendable}, $G$ is extendable.
\end{proof}

As a consequence of Theorem \ref{regular} and Lemma \ref{srg3} we obtain the following:
\begin{theorem}\label{regular2}
Let $G$ be a graph of order $n$ with largest Seidel eigenvalue $2\theta-1$.
Let $r =  \rk(S(G) -(2\theta-1) \mathbf{I})$. Assume that $n = \frac{r(r+1)}{2}$. Then $G$ is strongly maximal.
\end{theorem}

We see that $L(K_8)$ attains the absolute bound for $r=7$,
and by Theorem~\ref{thm:3}, it is a unique strongly maximal graph with largest Seidel eigenvalue $\lambda = 3$ up to switching.
We show a similar result for $(r,\lambda) \in \{(2,2),(3,\sqrt{5})\}$ at the end of this section.
\begin{proposition}	\label{prop:2,5}
	If a graph with largest Seidel eigenvalue $2$ (resp.\ $\sqrt{5}$) is strongly maximal,
	then it is switching equivalent to $\overline{K_3}$ (resp.\ $C_5 + K_1$).
\end{proposition}

The only other graph known to attain the absolute bound (for $r=23$) is the disjoint union of the McLaughlin graph and $K_1$.
We pose the following questions.
\begin{question}
	Is a strongly maximal Seidel matrix with largest eigenvalue $5$ unique up to switching?
\end{question}

\begin{question}
	Does there exist a strongly maximal graph with largest Seidel eigenvalue $2t+1$ where $t \geq 1$ is an integer?
\end{question}

We remark that Proposition~\ref{prop:empty} asserts that for every positive integer $t$, the empty graph $\overline{K_{2t+1}}$ is a strongly maximal graph with largest Seidel eigenvalue $2t$.

\begin{definition}
	Let $n$ be an integer at least $3$, and let $\lambda(n)$ be the minimum value of the largest  Seidel eigenvalues of graphs of order $n$ not switching equivalent to a complete graph.
\end{definition}
As a direct consequence of this definition, the sequence $(\lambda(n))_{n=3}^{\infty}$ is weakly increasing.
In other words, if a graph has largest Seidel eigenvalue less than $\lambda(n)$,
then it is of order less than $n$ or switching equivalent to a complete graph.	
Since the graphs of small orders are easily determined up to switching (see~\cite[TABLE~4.1 and TABLE~4.2]{Lint1966}),
we can verify that $\lambda(3)=2$, $\lambda(4) = \lambda(5)= \lambda(6)=\sqrt{5}$ and $\lambda(7)=(-3+\sqrt{65})/2 \approx 2.53$.
%$lambda(8)=2.5826\codts,\lambda(9)=2.6236$,
Using the value $\lambda(7)$, we can show Proposition~\ref{prop:2,5} as follows.

\begin{proof}[Proof of Proposition~\ref{prop:2,5}]
	Every graph with largest Seidel eigenvalue less than $\lambda(7) \approx 2.53$ is of order at most $6$ or switching equivalent to a complete graph.
	Hence every strongly maximal graph with largest Seidel eigenvalue $2$ or $\sqrt{5}$ is of order at most $6$.
	Checking the graphs of order at most $6$, we obtain the desired result.
\end{proof}

Note that $\lambda(n) \leq 3$ since $L(K_{2,n})$ has largest Seidel eigenvalue $3$ (see Theorem~\ref{thm:3}).
In the next section, we will discuss the behavior of $\lambda(n)$ to study strongly maximal graphs with largest Seidel eigenvalue less than $3$.

\section{Infinitely many strongly maximal graphs}	\label{sec:ex}

In this section, we discuss the existence of infinitely many strongly maximal graphs with largest Seidel eigenvalue less than $3$,
and provide two families of infinitely many strongly maximal graphs with unbounded largest Seidel eigenvalue.

\subsection{Strongly maximal graphs with largest Seidel eigenvalue less than $3$}
We determined strongly maximal graphs with largest Seidel eigenvalue $\lambda \in \{2,\sqrt{5},3\}$ in Sections~\ref{sec:3} and \ref{sec:strong_max}.
In this subsection, we show the following propositions to treat the case of $\lambda < 3$ more thoroughly.

Recall that for an integer $n \geq 3$, the real number $\lambda(n)$ is the minimum value of the largest  Seidel eigenvalues of graphs of order $n$ not switching equivalent to a complete graph.
By investigating the behavior of the sequence $(\lambda(n))_{n=3}^\infty$, the existence of infinitely many strongly maximal graphs with largest Seidel eigenvalue less than $3$ is derived.

\begin{proposition}	\label{prop:lamda}
	For every integer $n \geq 3$, the value $\lambda(n)$ is less than $3$.
	Furthermore, the sequence $(\lambda(n))_{n=3}^\infty$ converges to $3$.
\end{proposition}

\begin{proposition}	\label{prop:13}
	For each real number $\lambda$ in the open interval $(1,3)$, the number of graphs with the largest Seidel eigenvalue $\lambda$ is finite.
	In particular, if $\lambda$ is the largest Seidel eigenvalue of a graph, then there exists a strongly maximal graph with largest Seidel eigenvalue $\lambda$.
\end{proposition}

To prepare for the proof of these propositions, we introduce a graph $\hat{K}_n$,
which is the graph on $n+1$ vertices consisting of a complete graph $K_n$ with one extra edge.
In other words, this is the line graph of the graph $T_n$ obtained by attaching a new vertex to a leaf of the claw $K_{1,n}$.
%Namely, $V(T_n)=\{0,\ldots,n+1\}$ and $E(T_n)=\{\{0,i\} \mid i \in \{1,\ldots,n\} \} \cup \{\{n,n+1\}\}$.

\begin{lemma}	\label{lem:eigenT}
	For an integer $n \geq 2$, the largest Seidel eigenvalue of $\hat{K}_n$ is in the open interval $(3-4/n,3)$.
\end{lemma}
\begin{proof}
	Let $n$ be an integer at least $2$.
	We write $V(\hat{K}_n)=\{1,\ldots,n+1\}$ and $E(\hat{K}_n) = \{\{i,j\} \mid 1 \leq i < j \leq n \} \cup \{\{n,n+1\}\}$.
	Then the quotient matrix of $S(\hat{K}_n)$ with respect to an equitable partition $\{\{1,\ldots,n-1\},\{n\},\{n+1\}\}$ is
	\[
		\begin{pmatrix}
			2-n & -1 & 1 \\
			1-n & 0 &-1 \\
			n-1 & -1 & 0
		\end{pmatrix}
	\]
	Hence the characteristic polynomial of $S(\hat{K}_n)$ is
	$
		(x-1)^{n-2}(x+1)f(x),
	$
	where $f(x) := x^2 + (n-2)x - 3n+1$.
	Since $f(3-4/n) < 0$ and $f(3) > 0$,
	the desired result holds.
\end{proof}

\begin{lemma}	\label{lem:GhasT}
	Let $n$ be an integer at least $8$.
	If a graph has largest Seidel eigenvalue in the open interval $(1,3)$,
	then it contains $\hat{K}_{\lceil n/2 \rceil}$ as an induced subgraph up to switching.
\end{lemma}
\begin{proof}
	Let $G$ be a graph with largest Seidel eigenvalue in the open interval $(1,3)$.
	Applying Lemma~\ref{lem:reduc_to_L} and setting $\sL := \Lambda(G)$,
	we obtain a supergraph $L \in [\sL]$ of $G$.
	Then, by Lemma~\ref{lem:rk}, we have $\rk \sL = \rk(3\bI - S(G))+1 = n+1$.
	Since $\sL$ is an irreducible root lattice of rank $n+1 \geq 9$,
	it is isometric to either $\sA_{n+1}$ or $\sD_{n+1}$.
	In addition, the largest Seidel eigenvalue of $L$ is greater than $1$ by~\cite[Corollary~2.5.2]{brouwer2011spectra}.
	Hence Lemma~\ref{lem:SC} implies that $\sL$ is of type $D$ and $[\sL] = [L(K_{2,n-1})]$ holds.
	Since this implies that $[L] = [L(K_{2,n-1})]$,
	without loss of generality we may assume that $G$ is an induced subgraph of $L(K_{2,n-1})$.
	Then the $n$ vertices of $G$ correspond to $n$ edges of $K_{2,n-1}$.
	Hence we find that the graph induced by these edges has an induced subgraph isomorphic to $T_{\lceil n/2 \rceil}$.
	Therefore $G$ has an induced subgraph isomorphic to $\hat{K}_{\lceil n/2 \rceil} = L(T_{\lceil n/2 \rceil})$.
\end{proof}

\begin{proof}[{Proof of Proposition~\ref{prop:lamda}}]
	Let $n$ be an integer at least $8$.
	By Lemma~\ref{lem:eigenT}, the largest Seidel eigenvalue of $L(T_{n-1})$ is less than $3$.
	Hence so is $\lambda(n)$.
	
	Next we take a graph $G$ of order $n$ having largest Seidel eigenvalue $\lambda(n)$.
	Then since $n \geq 8$ and $\lambda(n) < 3$,
	Lemma~\ref{lem:GhasT} implies that $G$ contains $\hat{K}_{\lceil n/2 \rceil}$ as an induced subgraph up to switching.
	By~\cite[Corollary~2.5.2]{brouwer2011spectra}, the largest Seidel eigenvalue of $G$ is at least that of $\hat{K}_{\lceil n/2 \rceil}$.
	Hence by Lemma~\ref{lem:eigenT}, we have
	$
		\lambda(n) > 3-4/\lceil n/2 \rceil \geq 3 - 8/n.
	$
\end{proof}

\begin{proof}[Proof of Proposition~\ref{prop:13}]
	We fix a real number $\lambda \in (1,3)$.
	By Proposition~\ref{prop:lamda}, there exists an integer $n \geq 8$ such that $\lambda < \lambda(n)$.
	Recall that a graph with largest Seidel eigenvalue less than $\lambda(n)$ is of order less than $n$ or switching equivalent to a complete graph.
	Since any complete graph has largest Seidel eigenvalue $1$,
	we see that every graph with largest Seidel eigenvalue $\lambda$ is of order less than $n$.
	This implies the desired conclusion.
\end{proof}

Recall that Propositions~\ref{prop:lamda} and \ref{prop:13} provide infinitely many strongly maximal graphs with largest eigenvalue less than $3$.
By the proof of Proposition~\ref{prop:2,5}, we can determine the strongly maximal graphs with largest Seidel eigenvalue in $\{2,\sqrt{5},(-1+\sqrt{33})/2,-1+2\sqrt{3}\} \subset (1,3)$.
However for any largest Seidel eigenvalue $\lambda \in (1,3)$ except these four values, we were not able to determine the strongly maximal graphs with largest Seidel eigenvalue $\lambda$.
\subsection{Strongly maximal graphs with unbounded largest Seidel eigenvalue}
We have discussed the existence of strongly maximal graphs with largest Seidel eigenvalue in $(1,3] \cup \{5\}$.
In this subsection, we provide two families of infinitely many strongly maximal graphs with unbounded largest Seidel eigenvalue, each of which has exactly two Seidel eigenvalues.
\begin{lemma}	\label{lem:spec_ext}
	Let $G$ be a graph of order $n$ with exactly two Seidel eigenvalues $\lambda$ and $\mu$ with respective multiplicities $m(\lambda)$ and $m(\mu)$, where $\lambda>\mu$.
	Let $H$ be its proper supergraph of order $n+1$ with largest Seidel eigenvalue $\lambda$.
	Then the Seidel spectrum of $H$ is $\{ [\lambda]^{m(\lambda)}, [\mu]^{m(\mu)-1} ,\theta,\tau \}$
	where
	\begin{align}
		\theta + \tau = \mu \quad\text{ and }\quad
		\theta\tau = -n. \label{lem:spec_ext:1}
	\end{align}
\end{lemma}
\begin{proof}
	By~\cite[Corollary~2.5.2]{brouwer2011spectra}, the eigenvalues of $S(G)$ interlace those of $S(H)$.
	Hence we see that $\lambda$ and $\mu$ are Seidel eigenvalues of $H$ whose multiplicities are at least $m(\lambda)$ and $m(\mu)-1$, respectively.
	By $\tr S(G) = \tr S(H) = 0$, $\tr S(G)^2 = n(n-1)$ and $\tr S(H)^2 = n(n+1)$,
	the desired conclusion follows.
\end{proof}

The following proposition gives infinitely many strongly maximal graphs with exactly two Seidel eigenvalues, which are irrational numbers.

\begin{proposition}	\label{prop:relative_ir}
	Let $G$ be a graph with exactly two Seidel eigenvalues.
	If a Seidel eigenvalue of $G$ is not an integer, then $G$ is strongly maximal.
\end{proposition}
\begin{proof}
	Let $\lambda$ and $\mu$ be the Seidel eigenvalues of $G$ with $\lambda > \mu$, and $n$ the order of $G$.
	Then, since two Seidel eigenvalues of $G$ are algebraically conjugate, the Seidel spectrum of $G$ is $\{[\lambda]^{n/2},[\mu]^{n/2}\}$.
	By way of contradiction, we assume that $G$ is extendable.
	Namely, there exists a supergraph $H$ of $G$ such that its order is $n+1$ and its largest Seidel eigenvalue is $\lambda$.
	By~Lemma~\ref{lem:spec_ext}, the Seidel spectrum of $S(H)$ is $\{[\lambda]^{n/2},[\mu]^{n/2-1},\theta,\tau\}$ for some $\theta$ and $\tau$.
	Since $\lambda$ and $\mu$ are algebraically conjugate, without loss of generality we may assume that $\mu=\tau$.
	This is impossible by~\eqref{lem:spec_ext:1}.
\end{proof}

\begin{example}
	Let $q$ be a prime power congruent to $1$ modulo $4$, and let $P(q)$ denote the Paley graph of order $q$.
	Then the Seidel spectrum of $P(q)+K_1$ is $\left\{[\pm\sqrt{q}]^{(q+1)/2}\right\}$.
	 If $q$ is not a square, then we may apply Proposition~\ref{prop:relative_ir} to $P(q)+K_1$, and conclude that $P(q) + K_1$ is strongly maximal.
\end{example}

Next the following proposition gives infinitely many strongly maximal graphs with exactly two Seidel eigenvalues, which are integers.

\begin{proposition}	\label{prop:empty}
	For a positive integer $n$, the empty graph $\overline{K}_n$ is extendable if and only if $n$ is even.
\end{proposition}
\begin{proof}
	By direct calculation, we see that the Seidel spectrum of $\overline{K}_n$ is $\{n-1,[-1]^{n-1}\}$.
	First we assume that $n$ is even, and prove that $\overline{K}_n$ is extendable.
	Since $\overline{K}_{n}$ is switching equivalent to $K_{t,t}$ where $t := n/2$,
	it suffices to show that $K_{t,t}$ is extendable.
	Note that the smallest eigenvalue of $K_{t,t}$ equals $-t$.
	Since
	\[
		p(K_{t,t})  = \frac{2t}{t + t} \leq 2-\frac{1}{t},
	\]
	Lemma~\ref{k-reg extendable} implies that $K_{t,t}$ is extendable.
	
	Next we assume that $n$ is odd, and prove that $\overline{K}_{n}$ is strongly maximal.
	By way of contradiction, we assume that $\overline{K}_{n}$ is extendable.
	Namely, there exists a supergraph $H$ of $\overline{K}_{n}$ such that its order is $n+1$ and its largest Seidel eigenvalue is $n-1$.
	By Lemma~\ref{lem:spec_ext}, the characteristic polynomial $\Psi_{S(H)}(x)$ of $S(H)$ satisfies that
	\begin{align*}
		\Psi_{S(H)}(x)
		&=(x-(n-1))(x+1)^{n-2}(x^2+x-n) \\
		&\equiv x(x+1)^{n-2}(x^2+x+1) \mod 2\mathbb{Z}[x].
	\end{align*}
	However, by~\cite[Lemma~2.2]{greaves2018}, we have
	\[
		\Psi_{S(H)}(x) \equiv
		(x+1)^{n+1} \mod 2\mathbb{Z}[x].
	\]	
	They contradict, and the desired result is derived.
\end{proof}

\section*{Acknowledgements}
\indent

We greatly thank Professor Min Xu for supporting M.-Y. Cao to visit University of Science and Technology of China.

J.H. Koolen is partially supported by the National Natural Science Foundation of China (No. 12071454), Anhui Initiative in Quantum Information Technologies (No. AHY150000) and the project ``Analysis and Geometry on Bundles'' of Ministry of Science and Technology of the People's Republic of China.

A. Munemasa is partially supported by the JSPS KAKENHI grant (JP20K03537).

K.~Yoshino is supported by a scholarship from Tohoku University, Division for Interdisciplinary Advanced Research and Education.

\bibliographystyle{plain}
%\bibliography{200708b}
\bibliography{switchingclassesofgraphs}

\end{document}